\documentclass[reqno]{amsart}

\usepackage{amsmath}
\usepackage{amssymb}
\usepackage{amsthm}
\usepackage{graphicx}

\long\def\beginpgfgraphicnamed#1#2\endpgfgraphicnamed{\includegraphics{#1}} 

\newtheoremstyle{plainsl}%
	{\topsep}
	{\topsep}
	{\slshape} 
	{}
	{\normalfont\bfseries}
	{.}
	{ }
	{}

\newtheorem{theorem}{Theorem}[section]
\newtheorem{lemma}[theorem]{Lemma}

\newcommand\cref[1]{Corollary~\ref{cor:#1}}

\def\sqr#1#2{{\vbox{\hrule height.#2pt
    \hbox{\vrule width.#2pt height#1pt \kern#1pt
        \vrule width.#2pt}\hrule height.#2pt}}}
\def\eqed{\sqr53}
\def\qed{%
    \ifmmode\eqno\eqed
    \else\nobreak\ \hfill\eqed\medbreak\fi}

\begin{document}

\title{Graphs whose flow polynomials have only integral roots}
\author{Joseph P.S. Kung}
\author{Gordon F. Royle}

\address{Department of Mathematics\\ University of North Texas\\
Denton TX 76203\\ USA}
\email{kung@unt.edu}

\address{School of Mathematics and Statistics
\\ University of Western Australia\\ Nedlands WA 6009\\ Australia}
\email{gordon@maths.uwa.edu.au}

\subjclass[2000] {Primary 05B35;
Secondary 05C15}

\keywords{Graphs, flow polynomials, flow roots, cographic matroids}


\begin{abstract}
We show if the flow polynomial of a bridgeless graph $G$ has only integral roots,
then $G$ is the dual graph to a planar chordal graph. We also show that for 3-connected cubic graphs, the same conclusion holds under the weaker hypothesis that it has only real flow roots. Expressed in the language of matroid theory, this result says that the  cographic matroids with only integral characteristic roots are the cycle matroids of planar chordal graphs.
\end{abstract}

\maketitle

\section{Introduction}

For each different type of polynomial associated with a graph or matroid,
a natural and usually well-studied question is to determine if and when
the polynomial factors completely over the integers, or equivalently,
has only integer roots. For example, consider the {\em chromatic polynomial} $P(G;\lambda),$ which is defined to be the number of ways of properly coloring the vertices of the graph $G$ with at
most $\lambda$ colors.  The {\em chromatic roots} of $G$ are the roots of the chromatic polynomial of $G,$  and it has been a long-standing open question to characterize the graphs with integral chromatic roots. Chordal graphs, which are defined to be
graphs with no induced cycles of length greater than $3$, have integral chromatic roots, but there are also many non-chordal graphs with this property (see \cite{MR1802609, MR1489869, MR1887942, MR2190791}), and a complete characterization seems difficult, and perhaps even impossible.

The polynomial dual to the chromatic polynomial is the {\em flow polynomial}
$F(G;\lambda),$ defined to be the number of nowhere-zero flows on the graph $G$ taking values in an abelian group of order $\lambda$ (see Tutte \cite{MR0061366},  Brylawski and Oxley \cite{MR1165543}). The roots
of $F(G;\lambda)$ are called the {\em flow roots} of $G,$ and in this paper we
characterize the graphs with integral flow roots. As the chromatic polynomial of a planar graph is the flow polynomial of its dual (up to a factor of a power of $\lambda$), the duals of planar chordal graphs provide obvious examples of graphs with integral flow roots. Using an inequality for coefficients of
polynomials with real roots, an algebraic argument (first used in \cite{MR1900006}) to extract information from the coefficients of the flow polynomial, and a product  formula
 from matroid theory, we show that a graph with integral flow roots is the dual of a planar chordal graph. Loosely speaking, our main result shows that the obvious examples are the {\em only} examples.

\begin{theorem} \label{graphmain}
If $G$ is a bridgeless graph, then its flow roots 
are integral if and only if $G$ is the dual of 
a planar chordal graph. 
\end{theorem}

We note that  Theorem~\ref{graphmain} implies the theorem of Dong and Koh \cite{MR1489869} that planar graphs with integral chromatic roots are chordal.

Simple planar chordal graphs have a very restricted structure.
A $2$-connected planar chordal graph is constructed by starting from a
triangle $K_3$ and then repeatedly joining a new vertex either to both ends of an edge or to three vertices of a triangular face. The operation of joining a new vertex to an edge creates a $2$-vertex-cutset which persists throughout any
subsequent operations, and so the graph is $3$-connected if and only if it
arises from the complete graph $K_4$ by repeatedly inserting a vertex of degree $3$ into a face. Thus the $3$-connected planar chordal graphs form a very special class of triangulations, in fact precisely the class of uniquely 4-colorable planar graphs (Fowler \cite{fowler98}). At the other extreme are the graphs obtained from a triangle by using only the {\em first} operation of joining a new vertex to an edge (i.e., never creating a $K_4$). These graphs are called {\em $2$-trees} and it is well known that they are maximal series-parallel graphs with respect to edge addition.

We develop and present our results in the more general context of matroid
theory because the chromatic and flow polynomials of graphs are just the
characteristic polynomials of specific classes of matroids, whereas the main
ideas in our proof apply in general. Furthermore, there are various other
natural classes of matroids where it may be possible to characterize the
matroids whose characteristic polynomials have only integer roots.
We briefly discuss questions and conjectures of this nature in Section~\ref{supersolvable}.


\section{Preliminaries}

Recall that the {\em characteristic polynomial} $\chi(M;\lambda)$ of a matroid $M$ is defined in the following way:
if $M$ has a rank-$0$ element, then $\chi(M;\lambda) = 0$ and if
$M$ has no rank-$0$ elements, it is defined by
$$
\chi(M;\lambda) = \sum_{X: X \in L(M)}
\mu(\emptyset,X) \lambda^{\mathrm{rank}{M} - \mathrm{rank}{X}},
$$
where $L(M)$ is the lattice of flats of $M$ and $\mu$ is its 
M\"obius function
(see \cite{MR0174487}).
We call the roots of $\chi(M;\lambda)$ the {\em characteristic roots} of $M.$
If $M$ has no rank-$0$ elements, then the characteristic polynomial of $M$
depends only on its lattice of flats.  The {\em simplification}
of a matroid $M$ is the matroid obtained from $M$ by removing all
rank-$0$ elements and deleting all but one element in each rank-$1$ flat.
The lattice of flats is unchanged
under simplification; hence, if a matroid starts off with no rank-$0$ elements,
the characteristic polynomial is also unchanged.

Chromatic and flow polynomials of graphs are special cases of characteristic polynomials of matroids:
indeed, $P(G;\lambda) = \lambda^c\chi(M(G);\lambda),$
where $M(G)$ is the cycle matroid of $G$ and $c$ is the number of connected components in $G,$ and
$F(G;\lambda) = \chi(M^{\perp}(G);\lambda),$ where $M^{\perp}(G),$ the cocycle
matroid of $G,$ is the dual of $M(G).$
In particular, note that the flow polynomial $F(G;\lambda)$
depends only on the simplification of $M^{\perp}(G).$

 A {\em cutset} $C$ in a graph $G$ is  a set of edges such that $G - C$ has more connected components than $G$. A {\em bridge} in a graph is a cutset of size $1$, and if $G$ has a bridge, then its flow polynomial is identically zero. In matroid terms, the cocycle matroid $M^\perp(G)$ has a rank-0 element and so its characteristic polynomial is identically zero. To avoid this degenerate case, we henceforth consider only bridgeless graphs. 

If $G$ has no bridges, but is disconnected or has a cut-vertex, 
then it is either the disjoint union of two smaller graphs $G^{\prime}$ and $G^{\prime\prime}$ or it is obtained by identifying a vertex of  $G^{\prime}$ with a vertex of $G^{\prime\prime}$. In either case, the flow polynomial of $G$ is determined purely by the flow polynomials
of $G^{\prime}$ and $G^{\prime\prime}$:
\begin{equation}\label{directsum}
F(G;\lambda) =F(G^{\prime};\lambda) F(G^{\prime\prime};\lambda).
\end{equation}
This situation causes no difficulty however because it is easy to see that if $G^{\prime}$ and $G^{\prime\prime}$ are the duals of planar chordal graphs, then so is $G$. In matroid terms, the cocycle matroid $M^\perp(G)$ is disconnected and equal to the direct sum $M^\perp(G^{\prime}) \oplus M^\perp(G^{\prime\prime})$. 

If $G$ is $2$-vertex-connected, but has a $2$-cutset, then its flow polynomial is
unchanged if one of the edges in the cutset is contracted, and this process can be repeated until the graph is $3$-edge-connected.
A vertex of degree $2$ necessarily yields a 2-cutset, but not all $2$-cutsets arise in this manner.
In matroid terms, any 2-cutset corresponds to a series pair in the cycle matroid $M(G)$ and hence a
parallel pair in the cocycle matroid $M^\perp(G)$. Therefore the process of repeatedly contracting an edge in a $2$-cutset until no $2$-cutsets
remain is just simplifying the cocycle matroid. It proves convenient for us to work with simple matroids,
but it is important to note that this implies the main result even if the original matroid is not simple.
To see this, suppose that the simplification of $M^\perp(G)$ is the cycle matroid
of a simple planar chordal graph $H$. Then $M^\perp(G)$ is the cycle matroid
of the graph obtained from $H$ adding some edges in parallel to existing edges.
As this process does not alter planarity or the property of being chordal,
the resulting graph is still planar and chordal, though no longer simple.

A crucial step in the proof of Theorem \ref{graphmain} is to show that certain minimal $3$-cutsets exist in any graph whose flow polynomial has integer roots.
Minimal $3$-cutsets have rank $2$ and are closed; hence,
they form a $3$-point line in $M^{\perp}(G).$ In a $3$-edge-connected graph $G$, we call a minimal $3$-cutset {\em proper} if its deletion separates $G$ into
disjoint subgraphs $G^{\prime}$ and $G^{\prime\prime}$, each containing
at least one edge. An improper $3$-cutset necessarily consists of the $3$
edges incident on a vertex $v$ of degree $3$.

Since deletion in the graph $G$ corresponds to contraction in the
cocycle matroid $M^{\perp}(G),$ a proper $3$-cutset $L$
induces the following (non-trivial) separation in the cocycle matroid:
$$
M^{\perp}(G \backslash L) =
M^{\perp}(G) \big/ L = M^{\perp}(G^{\prime})
\oplus M^{\perp}(G^{\prime\prime}).
$$
In turn, this separation induces a product formula for flow polynomials.

\begin{lemma} \label{flowcut}
Suppose that the graph $G$ has a minimal $3$-cutset $L.$  Let
$G_1$ and $G_2$ be the two graphs $G \big/ G^{\prime\prime}$ and
$G \big/ G^{\prime}$ obtained by contracting each side of the
cutset to a single vertex. Then
\begin{equation}\label{flow3cut}
F(G;\lambda) = \frac{F(G_1;\lambda) F(G_2;\lambda)}{(\lambda-1)(\lambda-2)}.
\end{equation}
\end{lemma}

\noindent
We remark that an improper $3$-cutset associated with a vertex $v$
may induce a non-trivial separation.  This
occurs if and only if when $v$ and its incident edges are deleted,
the resulting graph has a new cut-vertex.

\begin{figure}
\begin{center}
\beginpgfgraphicnamed{fig-flow3cut}
\begin{tikzpicture}
\tikzstyle{vertex}=[circle, fill=white, draw=black, inner sep = 0.07cm]
\node[vertex] (v0) at (0,0) {};
\node[vertex] (v1) at (2,0) {};
\node[vertex](v2) at (3.5,0){};
\node[vertex] (v3) at (3.5,2) {};
\node[vertex] (v4) at (2,2) {};
\node[vertex] (v5) at (0,2) {};
\node[vertex] (v6) at (3,1.5){};
\node[vertex] (v7) at (0.5,1.5){};
\node [vertex] (v8) at (1,1){};
\node [vertex] (v9) at (1.5,1){};
\draw (v1)--(v8)--(v9)--(v1);
\draw (v4)--(v8);
\draw (v4)--(v9);
\draw (v0)--(v7)--(v4);
\draw (v5)--(v7);
\draw (v0)--(v1)--(v2)--(v3)--(v4)--(v5)--(v0);
\draw (v2)--(v6)--(v4);
\draw (v3)--(v6);
\draw (v1)--(v4);

\draw [thick, dashed] (2.5,-0.25)--(2.5,2.25);

\pgftransformxshift{5cm}
\node[vertex] (v0) at (0,0) {};
\node[vertex] (v1) at (2,0) {};
\node[vertex](v2) at (3,1){};
\node[vertex] (v3) at (3,1) {};
\node[vertex] (v4) at (2,2) {};
\node[vertex] (v5) at (0,2) {};
\node[vertex] (v6) at (3,1){};
\node[vertex] (v7) at (0.5,1.5){};
\node [vertex] (v8) at (1,1){};
\node [vertex] (v9) at (1.5,1){};
\draw (v1)--(v8)--(v9)--(v1);
\draw (v4)--(v8);
\draw (v4)--(v9);
\draw (v0)--(v7)--(v4);
\draw (v5)--(v7);
\draw (v0)--(v1)--(v2)--(v3);
\draw (v4)--(v5)--(v0);
\draw (v2)--(v6);
\draw (v4) .. controls (2.7,1.7) and (2.7,1.7) .. (v6);
\draw (v3) .. controls (2.3,1.3) and (2.3,1.3) .. (v4);
\draw (v3)--(v6);
\draw (v1)--(v4);
\end{tikzpicture}
\endpgfgraphicnamed
\end{center}
\caption{Graph $G$ with a 3-edge cutset and corresponding graph $G_1$}
\label{fig-3cutset}
\end{figure}
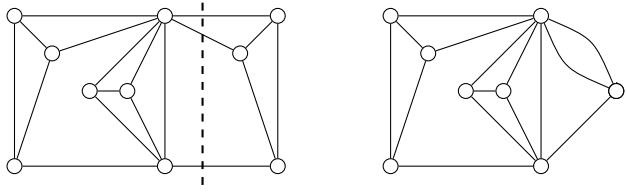

The formula in Lemma~\ref{flowcut} is a special case of a matroid formula of Brylawski for generalized parallel connection over a modular flat (see
\cite{MR0357163} and \cite{MR1165543}, p.~205), and it appeared in this form as Lemma 29 in Jackson's survey \cite{MR2005532}. Figure~\ref{fig-3cutset} shows a graph $G$ with a minimal 3-edge cutset, and the graph $G_1$ formed by contracting one side of the cut. In this case, $G_2$ is the graph $K_4$ which has flow polynomial $(\lambda-1)(\lambda-2)(\lambda-3)$, and 
\begin{equation*}
F(G;\lambda) = F(G_1;\lambda) (\lambda-3) = (\lambda-1)(\lambda-2)^3(\lambda-3)^2(\lambda^3-5\lambda^2+9\lambda-7).
\end{equation*}

The reader may recognize the formula \eqref{flow3cut} as the flow analogue of the formula for the chromatic polynomial of the ``join'' of two graphs at a triangle, which has a simple counting proof. 
 Lemma~\ref{flowcut} can be proved in several ways, in particular, by a routine contraction-and-deletion argument.  It is a little harder to give a direct counting argument analogous to the chromatic polynomial case, but this can be done by considering flows with values in the direct product $\mathbb{Z}_2^m$ of $m$ copies of the integers modulo $2,$ exploiting the fact that the flow polynomial depends only on the order of the group and not its structure.
 
\section{Polynomials with only real roots}

We will use the following easy result about polynomials. This simple lemma is surely known but we are unable to find a reference.

\begin{lemma}
\label{rootbound}
Let
$$
p(\lambda)
= \lambda^n - a_1 \lambda^{n-1} + a_2 \lambda^{n-2} + \cdots
+ (-1)^m a_m \lambda^{n-m} + \cdots + (-1)^n a_n
$$
be a polynomial of degree $n$ with positive real roots
$\lambda_1, \lambda_2, \ldots, \lambda_n$ and let
$$
\bar{\lambda} =
\frac {\lambda_1 + \lambda_2 + \cdots + \lambda_n}{n} =  \frac {a_1}{n}.
$$
Then
$$
a_m \le \binom {n}{m} \bar{\lambda}^{m}.
$$
Equality occurs for any one index $m,$ where $2 \le m \le n,$
if and only if $p(\lambda) = (\lambda - \bar{\lambda})^n.$
\end{lemma}

\begin{proof}
The coefficient $a_m$ is given by the elementary symmetric function of degree $m$ evaluated at the roots:
\begin{equation*}
a_m = e_m(\lambda_1, \lambda_2, \ldots, \lambda_n) = \sum_{1 \leq i_1 <  i_2 < \cdots < i_m \leq n} \lambda_{i_1} \lambda_{i_2} \cdots \lambda_{i_m}.
\end{equation*}

To prove the lemma, it suffices to show that if $m \geq 2$ and two of the roots, say $\lambda_1$ and $\lambda_2$, are not
equal and $\nu = \frac {1}{2} (\lambda_1 + \lambda_2),$  then
$$
e_m(\nu,\nu,\lambda_3, \ldots,\lambda_n)
> e_m(\lambda_1,\lambda_2,\lambda_3, \ldots,\lambda_n).
$$ 
There are three kinds of terms in the sum for
$e_m(\lambda_1,\lambda_2,\ldots,\lambda_n).$
The first are those terms not containing $\lambda_1$ or $\lambda_2$, which are unchanged when we change $\lambda_1$ and $\lambda_2$ to $\nu.$
The second are those terms containing exactly one of $\lambda_1$ or $\lambda_2.$
These terms come in pairs,
$ \lambda_1 \lambda_{i_2} \cdots \lambda_{i_m} $ and
$ \lambda_2 \lambda_{i_2} \cdots \lambda_{i_m}. $
The sum of the two terms in each pair is unchanged when $\nu$ is substituted for
$\lambda_1$ and $\lambda_2.$ The third kind of terms are those containing both $\lambda_1$ and $\lambda_2.$
Then by the arithmetic-geometric mean inequality,
$$
\nu^2 =   \frac {(\lambda_1 + \lambda_2)^2}{4} \ge \lambda_1  \lambda_2,
$$
where the inequality is strict when $\lambda_1 \ne \lambda_2$. Therefore these terms strictly increase when $\nu$ is substituted for $\lambda_1$ and $\lambda_2$.
\end{proof}

The proof of Lemma~\ref{rootbound} can be easily adapted to prove a variation.

\begin{lemma}\label{rootbound2}
Suppose that the polynomial $p(\lambda) $ in the previous lemma has
$n$ positive integer roots
$\lambda_1, \lambda_2, \ldots, \lambda_n$ such that $\bar{\lambda}$ is not an integer.
Let $\lambda_* = \lfloor \bar{\lambda} \rfloor,$
$\lambda^* = \lceil \bar{\lambda} \rceil,$ and $\delta$ be the positive integer 
such that
$$
n \bar{\lambda} =
(n- \delta) \lambda^* + \delta \lambda_*.
$$
Then the coefficient $a_m$ is at most the value of the degree-$m$
symmetric function
evaluated with $n- \delta$ variables set to $\lambda^*$ and $\delta$ variables set to $\lambda_*.$
In particular,
$$
a_2 \le \binom {n- \delta}{2} {\lambda^*} ^2 +
(n- \delta ) \delta \lambda^* \lambda_* +
\binom {\delta}{2} \lambda_*^2.
$$
Equality occurs if and only if $p(\lambda) =
(\lambda - \lambda^*)^{n- \delta}  (\lambda - \lambda_*)^{\delta}.$
\end{lemma}

\begin{proof}  
If a root $\lambda_1$ is strictly less than $\lambda_*,$ then
there is a root $\lambda_2$ such that $\lambda_2 > \lambda^*.$
If $\lambda_1 + \lambda_2$ is an even integer, then we can replace
$\lambda_1$ and $ \lambda_2$ by two roots, both equal to the average
$\frac {1}{2}(\lambda_1 + \lambda_2).$
If $\lambda_1 + \lambda_2$ is an odd integer, then we replace
$\lambda_1$ and $ \lambda_2$ by $\lfloor \frac {1}{2} (\lambda_1 + \lambda_2)\rfloor$ and
$\lceil \frac {1}{2} (\lambda_1 + \lambda_2) \rceil,$ the two integers
straddling the average.

We can now adapt the argument in the proof of Lemma 2.1, using a variation on the arithmetic-geometric mean inequality.  
\end{proof}

\section{Three-element circuits}

In this section we find lower bounds on the number of $3$-element circuits
in a matroid whose characteristic polynomial has only real or integer
roots.   Note that the size $3r-3$ that occurs in both lemmas
in this section is the maximum number of elements in a simple
cographic rank-$r$ matroid. Recall that a {\em line} in a matroid is a rank-2 flat.

\bigskip

\begin{lemma}\label{threecircuits}
Let $M$ be a simple matroid of rank $r$ with $3r-3$ elements.
Suppose that the lines in $M$ have at most three elements,
the characteristic roots of $M$ are real
and $\chi(M;2) = 0.$  Then $M$ has at least $3r-5$ $3$-element
circuits (or $3$-point lines). In addition, $M$ has exactly $3r-5$ $3$-element circuits if and only if
$$
\chi(M;\lambda) = (\lambda - 1) (\lambda - 2) (\lambda - 3)^{r-2}.
$$
\end{lemma}

\begin{proof}
Write
\begin{equation*}
\chi(M;\lambda) = (\lambda-1)(\lambda-2)\chi^{\dagger}(M;\lambda),
\end{equation*}
where 
\begin{equation*}
\chi^{\dagger}(M;\lambda) = \lambda^{r-2} - b_1 \lambda^{r-3} + b_2 \lambda^{r-4} - \cdots + (-1)^{r-2}b_{r-2}.
\end{equation*}
Let $\gamma_i$ be the number of lines in $M$ with $i$ elements, so $\gamma_3$ is the number of 3-element circuits. By hypothesis, $\gamma_i = 0$ if $i \ge 4.$ From standard results on characteristic polynomials of matroids, the coefficients of $\lambda^{r-1}$ and $\lambda^{r-2}$ in $\chi(M;\lambda)$ are equal to the number of elements $e$ and $\binom{e}{2} -\gamma_3$ respectively. Thus we have
\begin{align*}
b_1 + 3 & = 3r-3,\\
b_2 + 3b_1 + 2 &= \binom{3r-3}{2}-\gamma_3.
\end{align*}
and so
\begin{align*}
\gamma_3 &= \binom{3r-3}{2} - 9(r-2) - 2 - b_2
\end{align*}

If all the roots of $\chi^\dagger(M;\lambda)$ are real then, because $b_1 = 3(r-2)$, we can apply Lemma~\ref{rootbound} to $\chi^\dagger(M;\lambda)$ with $ \bar{\lambda} = 3$ and conclude that
\begin{equation}\label{eqn2}
b_2 \leq 9 \binom{r-2}{2}.
\end{equation}
On substituting this inequality into the equation given above
for $\gamma_3$, we conclude that
$$
\gamma_3 \ge 3r - 5.
$$
Equality occurs if and only if the inequality~\eqref{eqn2} is an equality,
that is,
when $\chi^\dagger(\lambda) = (\lambda - 3)^{r-2}.$
\end{proof}

For matroids with fewer elements, we can get
an analogous lower bound on the number of $3$-element circuits, but at
the cost of the stronger assumption that
the characteristic polynomial has {\it integer} roots,
rather than real roots.

\begin{lemma}\label{threecircuits2}
Let $M$ be a
simple rank-$r$ matroid with $3r-3-\delta$ elements,
where $0 \le \delta \le r-2.$
Suppose that the lines in $M$ have at most three elements,
the characteristic roots of $M$ are
integers, and $\chi(M;2) = 0.$  Then $M$ has
at least $ 3r-5- 2\delta$ $3$-element circuits (or $3$-point lines).
In addition, $M$ has exactly $3r-5- 2\delta$ $3$-element circuits if and only if
$$
\chi(M;\lambda) =
(\lambda - 1) (\lambda - 2)^{\delta + 1} (\lambda - 3)^{r-2-\delta}.
$$
\end{lemma}

\begin{proof}
Write
$$
\chi(M;\lambda) = (\lambda-1)(\lambda-2) \chi^{\dagger}(M;\lambda),
$$ and let $1$, $-b_1$, $b_2$ denote the leading
coefficients of $\chi^\dagger(M;\lambda)$ (as in the proof of the previous lemma).
Then we have
\begin{align*}
b_1 + 3 & = 3r-3-\delta,\\
b_2 + 3b_1 + 2 &= \binom{3r-3-\delta}{2}-\gamma_3,
\end{align*}
and so
\begin{equation}\label{gamma3}
\gamma_3  = \binom{3r-3-\delta}{2} - b_2 - 3(3r-6-\delta) - 2.
\end{equation}

If all the roots of $\chi(M;\lambda)$ are integers, then we can apply Lemma~\ref{rootbound2} to $\chi^{\dagger}(M;\lambda)$, where 
\begin{equation*}
(r-2) \bar{\lambda}=3 (r-2) - \delta = 3(r-2-\delta) + 2\delta
\end{equation*}
and so $\lambda_* = 2$ and $\lambda^* = 3$.  This yields the inequality
\begin{equation}\label{b2bound}
b_2 \leq 9 \binom{r-2-\delta}{2} + 6 \delta(r-2-\delta) + 4 \binom{\delta}{2},
\end{equation}
and substituting this into \eqref{gamma3} and canceling terms we obtain
$$
\gamma_3 \ge 3r - 5 - 2 \delta.
$$
Equality occurs if and only if the inequality \eqref{b2bound} is an equality, that is,
when $\chi^{\dagger}(M;\lambda)
= (\lambda - 3)^{r-2-\delta}(\lambda-2)^{\delta}.$
\end{proof}

The proof of Lemma \ref{threecircuits} can be used to prove the following
general result.

\vskip .33in \noindent
\begin{lemma}\label{threecircuits3}
 (a) Let $c \ge 2$ and $M$ be a
rank-$r$ connected simple matroid with $c(r-2) + 3$ elements with real
characteristic roots such that $\chi(M;2) = 0.$  Then
$$
\sum_{i:\, i \ge 3} \binom {i-1}{2} \gamma_i \ge \frac {c(c-1)}{2} (r-2) + 1.
$$
In particular, if all the lines in $M$ have at most $3$ points,
then $M$ has at least $\binom {c}{2}(r-2) + 1$
$3$-element circuits.  Equality occurs if and only if
$$
\chi(M;\lambda) = (\lambda - 1)(\lambda - 2)(\lambda - c)^{r-2}.
$$

(b) Let $M$ be a
rank-$r$ connected simple matroid with $c(r-1) + 1$ elements and real
characteristic roots.  Then
$$
\sum_{i:\, i \ge 3} \binom {i-1}{2} \gamma_i \ge \frac {c(c-1)}{2}(r-1).
$$
Equality occurs if and only if
$$
\chi(M;\lambda) = (\lambda - 1)(\lambda - c)^{r-1}.
$$
\end{lemma}

\section{Graphs with integral flow roots}

In this section we apply the results of the previous section to
flow polynomials.
Let $G$ be a $3$-edge-connected graph with vertex set $V$ and edge set $E$,
$v_i$ be the number of vertices of degree $i,$ and $r$ be the rank of its
cocycle matroid $M^{\perp}(G).$  Then
\begin{eqnarray*}
|V|  =  \sum_{i:\,i\ge 3} v_i, \qquad
|E|  =  \sum_{i:\,i\ge 3} \frac {iv_i}{2},
\end{eqnarray*}
and
$$
r   =  |E| - |V| + 1 =
\left[ \sum_{i:\, i\ge 3} \frac {(i-2)v_i}{2} \right] + 1.
$$
Let $\delta$ be defined by
\begin{equation}\label{delta}
\delta = \sum_{i:\,i \ge 3} (i-3) v_i.
\end{equation}
Then
$$
|V| = 2r - 2 - \delta, \quad\mathrm{and}\quad |E| = 3r - 3 - \delta.
$$

We shall prove Theorem~\ref{graphmain} in the following form.

\begin{theorem} \label{main}
Let $M$ be a simple rank-$r$ cographic matroid.  Suppose that
the characteristic polynomial $\chi(M;\lambda)$ has only integral  roots.
Then there exists a planar chordal graph $H$ with dual $G$ such that
$$
M = M^\perp(G) = M(H).
$$
\end{theorem}

\begin{proof}
We prove the result by induction on the rank $r.$ The cographic matroids of
rank $3$ or less are planar graphic, and it is easy to check that the theorem
holds in this case.  Thus, we may assume that $r \ge 4.$

If $M$ is not connected and equals $M^{\prime} \oplus M^{\prime\prime},$
then $\chi(M;\lambda) = \chi(M^{\prime};\lambda) \chi(M^{\prime\prime};\lambda).$
Hence, $\chi(M^{\prime};\lambda)$ and $\chi(M^{\prime\prime};\lambda)$
have integer roots and we can apply induction.   We can now suppose
that $M$ is simple and connected, and thus, we can find a $2$-vertex-connected, $3$-edge-connected
graph $G$ with vertex set $V$ and edge set $E$ such that $M = M^{\perp}(G).$

Let $\delta$ be defined as in Equation~\eqref {delta}.
Using the result that that $(\lambda-1)^2$ divides $\chi(M;\lambda)$ if and only if
$M$ is not connected
(see, for example, \cite{MR0215744}), we may assume that $\chi(M;\lambda)$
has exactly one root equal to $1$ and all other roots integers greater than or equal
to $2.$  Thus, $|E| \ge 2r-1,$ that is,
$$
\delta \le r-2.
$$

We distinguish two cases: $\delta \le r-3$ and $\delta = r-2.$
Suppose first that $\delta \le r-3.$  By Lemma~\ref{threecircuits2} the
matroid $M$ has at least
$3r - 5 - 2 \delta$ $3$-circuits.
A circuit in the cographic matroid $M$ corresponds
to a minimal cutset in the graph $G,$ and so $G$ has at least
$3r-5-2\delta$ minimal $3$-cutsets.
Of these, there are $v_3$ cutsets
separating a vertex of degree $3$ from the subgraph on the other vertices.
Since
$$
v_3 = 2r - 2 - \delta - \sum_{i: \, i \ge 4} v_i,
$$
there are at least
$$
3r - 5 - 2 \delta - \left( 2r - 2 - \delta -
\sum_{i:\,i \ge 4} v_i\right) = r - 3 - \delta + \sum_{i:\,i \ge 4} v_i
$$
proper minimal $3$-cutsets.
Since $\delta > 0$ if and only if there is at least one vertex of degree
greater than $3,$ we conclude that $G$ has least one proper minimal $3$-cutset $L.$
Let $G_1$ and $G_2$ be the two graphs obtained from $G$ and $L$ as defined in
Lemma~\ref{flowcut} (and illustrated in Figure~\ref{fig-3cutset}).
By Lemma~\ref{flowcut}, the flow polynomials of $G_1$ and $G_2$ have only integral roots
and by induction, both $G_1$ and $G_2$ are the duals of planar chordal graphs.
The graph $G$ is obtained from $G_1$ and $G_2$ by
identifying the three edges incident with a vertex of $G_1$ with the three
edges incident with a vertex of $G_2$.
In the planar dual, this
corresponds to forming $G^\perp$ by identifying a triangular face of
$G_1^\perp$ with a triangular face of $G_2^\perp$. Identifying a face in
each of two planar graphs gives a planar graph, and identifying a clique in
each of two chordal graphs yields a chordal graph.  Hence, $G$
is the dual of a planar chordal graph.

To finish the proof,
we consider the case when $\delta = r-2.$ In this case,
$\chi(M;\lambda)  = (\lambda - 1)(\lambda - 2)^{r-1},$
the graph $G$ has $2r-1$ edges,
$r$ vertices, and by Lemma~\ref{threecircuits2}, has
at least $r-1$ minimal $3$-cutsets.  If $G$ has $r$ or more minimal $3$-cutsets,
then as in the first case, $G$ has a proper minimal $3$-cutset and we can apply
induction.  Thus, we may assume that $G$ has exactly $r-1$ minimal $3$-cutsets, none of which is proper.
It follows that $r-1$ vertices in $G$ have degree $3,$ and so if $d$ is the degree of the last vertex then
$$
2|E| = 2(2r-1) = 3(r-1) + d,
$$
and so $d = r+1$. By Lemma~\ref{2tree}, $G$ is the dual of a $2$-tree and
hence dual to a planar chordal graph.
\end{proof}

\begin{lemma}\label{2tree}
A $2$-vertex-connected graph $G$ with $r$ vertices of which $r-1$ have
degree $3$ and one has degree $r+1$ is the planar dual of a $2$-tree.
\end{lemma}

\begin{proof}
We prove this by induction on $r$. When $r=2,$ the graph is a triple-edge
which is the planar dual of $K_3$.  So suppose that $r>2$ and denote the vertex of degree
$r+1$ by $v.$   As $G$ is $2$-vertex-connected, $v$ is not connected by a triple-edge
to any other vertex, but as $r+1 > r-1$, it must be joined by a double-edge to
some vertex $u$ where $u$ has a single further neighbor that we denote $w$.
The graph obtained by deleting the double-edge and identifying $u$ and $v$
is $2$-vertex-connected and has $r-1$ vertices, of which $r-2$ have degree $3$
and one has degree $r$, and hence by induction it is the planar dual of
a $2$-tree $T$. It is straightforward to see that adding a new vertex in the
face of size $r$ adjacent to the edge $\{v,w\}$ of $T$
(under the convention that edges of a graph are identified with those of
its planar dual)
yields a $2$-tree whose dual is $G$.
\end{proof}

Lemma~\ref{2tree} shows that when $\delta = r-2,$ the matroid $M$ 
is the cycle matroid of a maximal series-parallel graph.  
It might be useful to give an alternative (but equivalent) argument more congenial to matroid theorists.
As in the proof of the lemma, one shows that there is a vertex $u$ in $G$
of degree $3$ incident on a double-edge, which we label $a$ and $b,$
and a single edge, which we label $c.$
Then $\{a,b\}$ is a cocircuit of size $2$ and $\{a,b,c\}$ is a $3$-point line.
Let $X$ be the copoint complementary to the cocircuit $\{a,b\}.$  Then
the matroid $M$ is the parallel connection of the restriction $M|X$ and the line
$\{a,b,c\}$ at the point $c.$  By induction, $M$ is a parallel connection
of $r-1$ $3$-point lines, that is, the cycle matroid of a maximal series-parallel graph.

Observe that if we assume that $M$ has $3r - 3$ elements or, equivalently,
the graph $G$ is a cubic graph, then we can use Lemma~\ref{rootbound}
instead of Lemma~\ref{rootbound2} to obtain the following result.

\begin{theorem} \label{main-real}
If a $3$-connected cubic graph
$G$ has real flow roots, then $G$ is the dual of a chordal planar triangulation.
\end{theorem}

We do not know whether Theorem~\ref{main} holds if the hypothesis is weakened
so that we only assume that all the roots of $\chi(M;\lambda)$ are real.
Using Lemma \ref{threecircuits3}(a), we can show that if
$M$ is a rank-$r$ cographic matroid with real characteristic roots and
$\delta < \sqrt{2(r-2)},$
then $M$ is the cycle matroid of a planar chordal graph.

\section{Supersolvable Matroids}
\label{supersolvable}

In the remainder of this paper, we shall describe the matroid-theoretic aspects
of Theorem~\ref{main}.   We shall only consider matroids with no rank-$0$ elements.

Recall that a flat $X$ in a matroid $M$ is {\sl modular} if for every line $L$
in $M$ such that $\mathrm{rank}(X \vee L) = \mathrm{rank}(X) + 1,$
$X \cap L$ is non-empty (and hence a point or a rank-$1$ flat).
If $X$ is a modular flat, then the characteristic polynomial
$\chi(M|X;\lambda)$ of the restriction of $M$ to $X$ divides the
characteristic polynomial of $\chi (M;\lambda)$  (see Stanley \cite{MR0295976}).
A rank-$r$ matroid $M$ is {\em supersolvable} if
there exists a maximal chain of modular flats $X_0, X_1, X_2, \ldots, X_r,$
with $X_{i-1} \subset X_i$ and $\mathrm{rank}(X_i) = i.$  A maximal chain
of modular flats forces a complete factorization of
$\chi(M;\lambda)$ over the integers.  Explicitly, the
characteristic roots are $|X_i| - |X_{i-1}|,$ $i = 1,2,\ldots,r.$

The next lemma describe how a modular copoint forces a complete subgraph.   The first part of the lemma holds for arbitrary matroids (Lemma 5.14 in \cite{MR1411690}).

\begin{lemma} \label{modular}
Let $X$ be a modular copoint with complement $D$ in a simple binary
matroid $M.$  If $\mathrm{rank}(D) = d,$ then $M$ contains
an $M(K_{d+1})$-submatroid.
In particular, if $M$ has no $M(K_5)$-submatroid, then
$\mathrm{rank}(D) \le 3,$ and $M$ is the parallel connection of the
restriction $M|X$ and a point (at the empty set), a $3$-point line (at a point),
an $M(K_4)$ (at a $3$-point line), or a Fano plane $F_7$ (at a $3$-point line).
\end{lemma}

\begin{proof}
Let $I$ be an independent set of size $d$ in $D.$ Since $X$ is modular,
every line outside $X$ meets $X$ at a point. Hence each pair of points in $I$
determines a point of intersection in $X,$ and because $I$ is independent,
these points are distinct. The $\binom {d}{2}$ points in $X,$ together with
the $d$ points in $I,$ form an $M(K_{d+1}).$

If $M$ has no $M(K_5)$-submatroid, then $d \le 3,$ and $D$ equals a point,
two points, three independent points, or a $4$-circuit, corresponding to
the four cases listed in the lemma.
\end{proof}

A simple chordal graph $G$ can be built beginning with a vertex and repeatedly adding a
new vertex and all edges from that vertex to a complete subgraph.  This
construction yields a maximal chain of modular flats and hence, the
cycle matroid $M$ of $G$ is supersolvable.  On the other hand, Lemma~\ref{modular}
implies that if the cycle matroid of a graph $G$ is supersolvable, then
$G$ is chordal. We can now restate Theorem~\ref{main} for matroids.

\begin{theorem} \label{mainmatroid}
Let $M$ be a simple cographic matroid.
The following conditions are equivalent.

\begin{enumerate}
\item  The characteristic roots of $M$ are all integers.
\item $M$ can be constructed by taking parallel connections of copies of points, $M(K_3)$'s, or $M(K_4)$'s at the empty set,
a point or a $3$-point line, with the restriction that no line in an $M(K_4)$
can be used more than once in a parallel connection.
\item  $M$ is the cycle matroid of a planar chordal graph.
\item  $M$ is supersolvable.
\end{enumerate}
\end{theorem}

\begin{proof}
To show that (1) implies (2), we use the induction argument in the proof of
Theorem~\ref{main}, using as hypothesis the description of $M$ as
a parallel connection.  Note that if we take parallel
connections of three $M(K_4)$'s at a common line, then we obtain
an $M(K_{3,3})$-submatroid, which cannot occur inside a cographic matroid.
That (2) implies (3), (3) implies (4), and (4) implies (1) follow from
results discussed earlier.
\end{proof}

Although we proved directly that cographic matroids with integer roots are the cycle matroids of planar chordal graphs, it can be useful to view this as the combination of two separate results:
\begin{enumerate}
\item A cographic matroid with integral characteristic roots is supersolvable.
\item A supersolvable cographic matroid is the cycle matroid of a planar chordal graph.
\end{enumerate}

The second of these results can be proved with a direct argument.
Kuratowski's theorem says that if a graph is not planar, then
it has a subgraph which is a series extension of 
$K_5$ or $K_{3,3}.$
Dualizing, we conclude that if $M$ is the cocycle matroid
of a non-planar graph, then there is some set of elements $X$ such that
the contraction $M/X$ is a parallel extension of
$M^\perp(K_5)$ or $M^\perp(K_{3,3})$. However,
these latter matroids are not supersolvable, and as contraction preserves supersolvability, it follows that $M$ itself is not supersolvable.
Hence, if $M$ is supersolvable and cographic, it is the cocycle matroid of some
planar graph, that is, $M$ is the cycle matroid of a planar graph, and hence
the cycle matroid of a chordal graph.

When considering generalizations of our results to other classes of matroids,
it is natural to consider the two questions separately, i.e., asking which other classes
of binary matroids have the property that only supersolvable matroids have
integral characteristic roots, and then separately characterizing the supersolvable
matroids in the class. Theorem~\ref{mainmatroid}
and Lemma~\ref{modular} suggest the conjecture that
a binary matroid with no $M(K_5)$-minor with integral characteristic roots
is supersolvable.
%
%

We end by mentioning that if we remove the restriction that the matroid
is binary, there are many non-supersolvable matroids with
integral characteristic roots.  
See, for example, \cite{MR1620842}.   In the context of matroids,
chromatic and flow polynomials would seem to be very special cases.


\subsection*{Acknowledgement}

The authors gratefully acknowledge the contribution of
Maria Monks who alerted us to an error in an earlier version of this paper.

\bibliographystyle{amsplain}
\bibliography{gordonmaster}

\end{document}